\documentclass[12pt]{amsart}
 \newtheorem{theo}{Theorem}[section]

 \newtheorem{cor}[theo]{Corollary}

 \newtheorem{lem}[theo]{Lemma}
 \newtheorem{prop}[theo]{Proposition}

 \theoremstyle{definition}
 
 \theoremstyle{remark}

\newtheorem{exam}[theo]{Example}

\newtheorem*{question*}{Question}
 
 \numberwithin{equation}{section}

\usepackage{amsmath}
\usepackage{amsfonts}
\usepackage{amssymb}
\usepackage{xypic}
\usepackage{graphicx}

\newcommand{\qf}[1]{\mbox{$\langle #1\rangle $}}
\newcommand{\pff}[1]{\mbox{$\langle\!\langle #1\rangle\!\rangle $}}
\newcommand{\qpf}[1]{\mbox{$\langle\!\langle #1]] $}}
\newcommand{\HH}{{\mathbb H}}
\newcommand{\ZZ}{{\mathbb Z}}

\newcommand{\FF}{{\mathbb F}}

\newcommand{\MM}{{\mathbb M}}

\newcommand{\ovl}[1]{\overline{#1}}

\newcommand{\bbperp}{\operatornamewithlimits{\mbox{\Large\boldmath $\perp$}}}

\renewcommand{\phi}{\varphi}

\DeclareMathOperator{\spn}{span}

\begin{document}

\title[Splitting of quaternions and octonions]{Splitting of
quaternions and octonions over purely inseparable extensions
in characteristic 2}

\author{Detlev W.~Hoffmann}
\thanks{The author is supported in part by DFG grant HO 4784/2-1
{\em Quadratic forms, quadrics, sums
of squares and Kato’s cohomology in positive characteristic}.}

\address{Fakult\"at f\"ur Mathematik,
Technische Universit\"at Dortmund,
D-44221 Dortmund,
Germany}
\email{detlev.hoffmann@math.tu-dortmund.de}

\date{}

\begin{abstract} We give examples of quaternion and octonion division algebras
over a field $F$ of characteristic $2$ that split over a purely inseparable
extension $E$ of $F$ of degree $\geq 4$ but that do not split over any
subextension of $F$ inside $E$ of lower exponent, or, in the case of
octonions, over any simple subextension of $F$ inside $E$.  Thus,
we get a negative answer
to a question posed by Bernhard M\"uhlherr and Richard Weiss.
We study this question in terms of
the isotropy behaviour of the associated norm forms.
\end{abstract}


\maketitle

\section{Introduction}
Recall that a quaternion (resp. octonion) algebra $A$ over a field $F$
is a an associative
$4$-dimensional (resp. nonassociative $8$-dimensional) composition algebra
whose associated norm form $n_A$ is a nondegenerate quadratic form
defined on the $F$-vector space $A$.  This norm form has the property
that $A$ is a division
algebra iff $n_A$ is anisotropic.  Furthermore, $n_A$ determines $A$ in the sense
that if $A'$ is another quaternion (resp. octonion) algebra, then
$A\cong A'$ as $F$-algebras iff $n_A\cong n_{A'}$ (i.e., the norm forms
are isometric).  There is a vast amount of literature on such algebras.
In this note, we consider only base fields of characteristic $2$, a case
that, compared to characteristic not $2$,
gives rise to additional subtleties that make the theory in some sense
richer but also more intricate and complicated.   For example, 
when conisdering subfields of such algebras one has to distingush between
separable and inseparable quadratic extensions.  The starting point for this
investigation was a question in this context posed to me by Bernhard M\"uhlherr
und Richard Weiss.  They asked 
the following:  Given an octonion division algebra $O$ over a field $F$
of characteristic $2$ that splits over some purely inseparable algebraic extension
$E/F$, is it true that then $E$ contains a subfield $K$ that is a quadratic
extension of $F$ and that embeds into $O$?  We will show that the answer
is negative in general by constructing various counterexamples.

Our approach will be to consider
this question in the context of quadratic forms using the fact that
the division property of such algebras can be expressed in terms of the
anisotropy of the associated norm forms that are given by $2$-fold
Pfister forms in the case of quaternion algebras and by $3$-fold Pfister forms
in the case of octonion algebras.  This then allows to broaden the context
to study the isotropy behaviour of arbitrary $n$-fold Pfister forms
over purely inseparable extension in characteristic $2$.

In the next section, we will collect those facts about the theory of
bilinear and quadratic forms in characteristic $2$ that we will need
in our constructions.  In \S\,3, we will study the isotropy
behaviour of bilinear Pfister forms under purely inseparable extension
and construct various examples that we will built upon in \S\,4, where
we investigate the isotropy behaviour of quadratic Pfister forms
under purely inseparable extensions.  In that section, we will
construct our main examples.  Namely, given any integers $m\geq 2$,
$n\geq 3$ (resp. $n\geq 2$)
and $\ell$ with $1\leq \ell\leq \max\{ 1,m-2\}$ (resp.
$2\leq \ell\leq m$), we will construct a field $F$ with
an anisotropic $n$-fold quadratic Pfister form $\pi$ and a purely
inseparable extension $M/F$ such that $M$ has exponent $\ell$ over 
$F$ and $[M:F]=2^m$ and such that $\pi$ will become isotropic 
over $M$ but $\pi$ will stay anisotropic over any simple (resp.
exponent $\leq \ell-1$) extension of $F$ contained in $M$.
In \S\,5, we will rephrase these constructions in the original
context of quaternion and octonion algebras.

\section{Quadratic and bilinear forms in characteristic $2$}
From now on, all fields that we will consider will be of characteristic $2$,
all bilinear forms will be symmetric, nondegenerate and finite-dimensional,
and all quadratic forms will be finite-dimensional (but not necessarily
nondegenerate).  We refer to \cite{ekm} for any terminology or any facts that we
state without further reference.  As usual,
we use the symbols $\cong$, $\perp$ to denote
isometry and orthogonal sum,  and $\otimes$ for the tensor product of two bilinear
forms or of a bilinear with a quadratic form.
Let $q$ (resp. $b$) be a quadratic (resp. bilinear form) on an $F$-vector space
$V$.  The value set is defined to be $D_F(q)=\{ q(x)\,|\,x\in V\}\cap F^*$
(resp. $D_F(b)=\{ b(x,x))\,|\,x\in V\}\cap F^*$), and $q$ (resp. $b$) is said
to be isotropic if there exists some $x\in V\setminus\{ 0\}$ with
$q(x)=0$ (resp. $b(x,x)=0$).  If $(q',V')$ is another quadratic form,
then we say that $q$ represents $q'$ if there is a injective $F$-linear map
$\sigma:V'\to V$ such that $q(\sigma(x))=q'(x)$ for all $x\in V'$, and we write
$q'\prec q$.  Similarly, one defines $b'\prec b$ for another bilinear
form $(b',V')$.

Every quadratic form $q$ decomposes as 
$$q\cong [a_1,b_1]\perp\ldots\perp [a_r,b_r]
\perp\qf{c_1,\ldots,c_s},\ a_i,b_i,c_i\in F,$$
where $[a,b]$ stands for the quadratic form $ax^2+xy+by^2$ and 
$\qf{c_1,\ldots,c_s}$ corresponds to the restriction of $q$
to its radical and stands for the form $c_1x_1^2+\ldots +c_sx_s^2$.
In particular, the values $s,r$ in such a decomposition are uniquely determined
and $q$ is nondegenerate (i.e., $q$ has trivial radical) iff $s=0$.

If $r=0$, we say that the above $q$ is totally singular.  In this case,
$D_F(q)\cup\{ 0\}=\sum_{i=1}^sF^2c_i$ is a finitely generated subvector space
of $F$ (considered as vector space over $F^2$).  In fact, if 
$q\cong\qf{c_1,\ldots,c_s}$ and $q'\cong\qf{d_1,\ldots,d_t}$ are
both totally singular, then $q\cong q'$ iff $s=t$ and
$D_F(q)\cup\{ 0\}=D_F(q')\cup\{ 0\}$ iff $s=t$ and 
$\spn_{F^2}(c_1,\ldots,c_s)=\spn_{F^2}(d_1,\ldots,d_r)$.
We refer the reader to \cite{hl1} concerning further properties of
totally singular quadratic forms.

$\HH=[0,0]$ is called a quadratic hyperbolic plane, and $q$ is said to be
hyperbolic iff $q$ is isometric to an orthogonal sum of hyperbolic
planes.  If $q'$ is nondegenerate, then $q'\prec q$ iff there exists 
a quadratic form $q''$ with $q\cong q'\perp q''$.  
Also, it is not difficult to show that if $q$ is a nondegenerate quadratic
form (or a quadratic form with anisotropic radical), then
$q$ is isotropic iff $\HH\prec q$.

A $2$-dimensional isotropic bilinear form is called a metabolic plane.
For a suitable basis, the Gram matrix will be of shape
$\left(\begin{smallmatrix} 0 & 1 \\ 1 & a\end{smallmatrix}\right)$
for some $a\in F$, and
this metabolic form will be denoted by $\MM_a$, and $\MM_0=\HH_b$
will be called a bilinear hyperbolic plane.  A bilinear form $b$ is
called metabolic (hyperbolic) iff it is isomteric to an orthogonal sum of
metabolic (bilinear hyperbolic) planes.  A bilinear form is diagonalizable
iff it is not hyperbolic, in which case we use the notation
$\qf{a_1,\ldots,a_n}_b$ when $\{ e_1,\ldots,e_n\}$ is an
orthogonal basis of $(b,V)$ with $b(e_i,e_j)=\delta_{ij}a_i\in F$.

An $n$-fold bilinear Pfister form  ($n\geq 1$)
is a bilinear form that is isometric to
$\qf{1,a_1}_b\otimes\ldots\otimes\qf{1,a_n}_b$ for suitable $a_i\in F^*$,
in which case we write $\pff{a_1,\ldots,a_n}_b$ for short.
$\qf{1}_b$ is considered to be the unique $0$-fold bilinear Pfister form.
The set of forms isometric to $n$-fold bilinear Pfister forms over $F$ will be
denoted by $BP_nF$.

An $n$-fold quadratic Pfister form  ($n\geq 1$)
is a (nondegenerate) quadratic form that is isometric to
$\beta\otimes [1,a]$ for some $a\in F$ and some $\beta\in BP_{n-1}F$.
If $\beta\cong\pff{a_1,\ldots,a_{n-1}}_b$, this quadratic Pfister form
will be written
$\qpf{a_1,\ldots,a_{n-1},a}$ for short, and $P_nF$ will denote the set
of forms isometric   to $n$-fold quadratic Pfister forms over $F$.

We also consider totally singular quadratic forms that are
so-called quasi-Pfister forms of type
$\pff{a_1,\ldots,a_n}\cong\qf{1,a_1}\otimes\ldots\otimes\qf{1,a_n}$, $a_i\in F$
(here, we allow $a_i=0$).  The set of quadratic forms isometric to
$n$-fold quasi-Pfister forms will be denoted by $QP_nF$.  Note
that $D_F(q)\cup\{ 0\}=F^2(a_1,\ldots,a_n)$ is just the field extension
of $F^2$ inside $F$ generated by the $a_i$, and that two forms in
$QP_nF$ are isometric iff these associated fields are the same.

If $b\in BP_nF$, then the quadratic form $q_b$ given by $q_b(x)=b(x,x)$
is a form in $QP_nF$.  It should be noted that in this way, nonisometric forms
in $BP_nF$ may yet give rise to isometric forms in $QP_nF$.

Pfister forms play a central role in the theory of bilinear/quadratic forms
and they are characterized by many nice properties.  For example,
a bilinear (resp. quadratic) $n$-fold Pfister form ($n\geq 1$) is
metabolic (resp. hyperbolic) iff it is isotropic.
In particular, if $\phi$ is a bilinear (quadratic) $n$-fold Pfister form
for some $n\geq 1$, and if $\psi$ is any other bilinear (quadratic)
form with $\psi\prec a\phi$ for some $a\in F^*$ and if in addition
$\dim\psi >\frac{1}{2}\dim\phi$ (in which case $\psi$ is called
a Pfister neighbor of $\phi$), then $\psi$ is isotropic iff $\phi$
is isotropic iff $\phi$ is metabolic (hyperbolic). 
Also a quadratic (resp. bilinear, resp. quasi) Pfister form $\pi$
has the property that $\pi\cong a\pi$ for some $a\in F^*$ iff
$a\in D_F(\pi)$, a property that is often referred to as
roundness.

We are interested in the isotropy behaviour under purely inseparable
algebraic extensions  (p.i.\ extensions for short).  We will recall
a few facts regarding these extensions in the next section.

\section{Bilinear forms over purely inseparable field extensions}
Recall that if $F$ is a field of characteristic $2$, then elements
$a_1,\ldots,a_n\in F$ are called $2$-independent if 
$[F^2(a_1,\ldots,a_n):F^2]=2^n$, and a subset $S\subset F$ is
$2$-independent if any finite subset of $S$ is $2$-independent.
A $2$-basis of $F$ is a $2$-independent subset $S$ with $F^2(S)=F$
(such a $2$-basis always exists).

Let $E/F$ be a p.i.\ extension. This means that to each $a\in E$ there
exists some integer $n\geq 0$ such that $a^{2^n}\in F$.  The smallest such
$n$ will be called the exponent of $a$, $\exp_2(a)=n$, and we define 
the exponent of $E$ to be $\exp_2(E/F)=\sup\{\exp_2(a)\,|\,a\in E\}$.
A p.i.\ extension $E/F$ will be of exponent $1$ iff there exists
a nonempty $2$-independent subset $S\subset F$ such that
$E=F(\sqrt{s}\,|\,s\in S)$.  A finite p.i.\ extension $E/F$ is
called modular iff it is isomorphic to a tensor product of simple
extensions over $F$ iff there exists a $2$-independent subset
$\{ s_1,\ldots, s_r\}$ ($r\geq 0$) and integers $n_1,\ldots,n_r\geq 1$
such that
\begin{equation}\label{eq1}
E=F\bigl(\sqrt[2^{n_1}]{s_1},\ldots,\sqrt[2^{n_r}]{s_r}\bigr).
\end{equation}
It is an easy exercise to show that if $[E:F]=2^m$ with $m\leq 2$
then $E/F$ is modular.  But there are already examples of nonmodular
extensions with $[E:F]=8$ (we will encounter such an extension
in Example \ref{ex2}).

\begin{prop}\label{bil-pf}
Let $\alpha\cong\pff{a_1,\ldots, a_n}_b\in BP_nF$ and put
$A=\{a_1,\ldots, a_n\}\subset F^*$.  Let $\beta$ be a
bilinear form over $F$.  Let $E/F$ be a modular finite p.i.\ extension
as in Eq. (\ref{eq1}) for some $2$-independent subset 
$S=\{ s_1,\ldots, s_r\}\neq\emptyset$.  Let $L=F(\sqrt{s_1},\ldots,\sqrt{s_r})$
and $\sigma\cong\pff{s_1,\ldots, s_r}_b$.
\begin{enumerate}
\item[(i)] $\alpha$ is anisotropic iff $A$ is $2$-independent.
\item[(ii)] The following statements are equivalent:
\begin{enumerate}
\item[(a)] $\beta_E$ is isotropic;
\item[(b)] $\beta_L$ is isotropic;
\item[(c)] $\sigma\otimes\beta$ is isotropic.
\end{enumerate}
\item[(iii)] $\alpha_E$ is anisotropic iff $A\cup S$ is $2$-independent. 
\end{enumerate}
\end{prop}
\begin{proof}
(i) See \cite[Lem.~8.1]{h0}.

(ii) See \cite[Th.~5.2]{h1}.

(iii) follows from (ii)(c) applied to $\alpha$ instead of $\beta$, together
with (i) applied to the bilinear Pfister form
$\sigma\otimes\alpha\cong\pff{s_1,\ldots, s_r,a_1,\ldots, a_n}_b$.
\end{proof}

$\alpha\cong\pff{a}_b$ is clearly anisotropic iff
$a\in F\setminus F^2$.  In this case, if $E/F$ is any field extension,
then $\alpha_E$ is isotropic iff $a\in E^2$ iff $F(\sqrt{a})\subset E$,
so $\alpha$ becomes isotropic over a quadratic extension of $F$ contained in $E$.
This will generally no longer be true once we consider anisotropic
bilinear Pfister forms of fold $n+2\geq 2$ as the following example shows.

\begin{exam}\label{ex1}
Let  $F$ be a field (of characteristic $2$ as always) with a $2$-independent set 
$\{ x,y,z,c_1,\ldots,c_n\}$.  For example, we may
take $\FF_2(x,y,z,c_1,\ldots,c_n)$, the rational function field in $n+3$ variables
over $\FF_2=\ZZ/2\ZZ$.  Let $E=F(\sqrt{z},\sqrt{xz+y})$ and 
$\pi=\pff{x,y,c_1,\ldots,c_n}_b\in BP_{n+2}F$.
Clearly, $\{ z ,xz+y,x,y,c_1,\ldots,c_n\}$ is $2$-dependent.
By Proposition \ref{bil-pf}, $\pi$ is 
anisotropic over $F$ and isotropic over $E$.  Of course, the latter can also
be checked directly by noting that $\pi\cong\qf{1,x,y,xy}_b\perp\qf{\ldots}_b$ and
$$1\cdot\sqrt{xz+y}^2+x\cdot\sqrt{z}^2+y\cdot 1^2=0.$$
We show that for any quadratic subextension $F\subset K\subset E$, we have
that $\pi_K$ is anisotropic.      Such a quadratic subextension
is of shape $K=F(a)$ with 
$$a=u\sqrt{z}+v\sqrt{xz+y}+w\sqrt{z^2x+yz}$$
for some $u,v,w\in F$,
not all equal to $0$.

By Proposition \ref{bil-pf}, $\pi_K$ is isotropic iff
$\pff{a^2,x,y,c_1,\ldots,c_n}_b$ is isotropic over $F$.  We know that
$\pff{z,x,y,c_1,\ldots,c_n}_b$ is anisotropic over $F$.  Now the (an)isotropy of these
bilinear Pfister forms is equivalent to the (an)isotropy of the associated
quasi-Pfister forms.  We show that for these quasi-Pfister forms, we will
have
$$\pff{a^2,x,y,c_1,\ldots,c_n}\cong\pff{z,x,y,c_1,\ldots,c_n},$$
over $F$, which in turn implies that $\pi_K$ is anisotropic. 

It cleary suffices to show that $\pff{a^2,x,y}\cong \pff{z,x,y}$,
which just means that
$$D_F(\pff{a^2,x,y})\cup\{ 0\} =F^2(a^2,x,y)= F^2(z,x,y)
=D_F(\pff{z,x,y})\cup\{ 0\}.$$
It then suffices to show that $z\in F^2(a^2,x,y)$ and $a^2\in F^2(z,x,y)$.
Now
$$a^2=(wz)^2x+v^2y+u^2z+v^2xz+w^2yz\in F^2(z,x,y).$$
Conversely, note that we do not have $u=v=w=0$, and since $x,y$ are $2$-independent,
we thus have $0\neq s=u^2+v^2x+w^2y\in F^2(a^2,x,y)$.  We put
$$r=a^2+(wz)^2x+v^2y=z(u^2+v^2x+w^2y)=zs\in F^2(a^2,x,y)$$
and get that $z=rs^{-1}\in F^2(a^2,x,y)$.
\qed\end{exam}

In the next example, we will exhibit an
anisotropic $n$-fold bilinear Pfister form ($n\geq 2$) that will
become isotropic over a nonmodular p.i.\ extension but which will
not become isotropic over any proper subfield.  Now
p.i.\ extensions $E/F$ with $[E:F]\leq 4$ are easily seen to be
modular (quadratic, biquadratic, or gotten through
adjoining a 4th root of some nonsquare).  The smallest possible
degree for a nonmodular extension will thus be $8$, and we will use such an
example that can be traced back to Sweedler \cite[Example 1.1]{sw}.

\begin{exam}\label{ex2}
Let $F$ be as in Example \ref{ex1}.  Let $E=F(\sqrt[4]{z},\sqrt{x\sqrt{z}+y})$.
Clearly, $\exp_2(E/F)=2$ and an easy check shows that $[E:F]=8$.  Comparing degree
and exponent shows that $E/F$ is not simple.  If $E/F$ were modular, then it would
contain a biquadratic subextension, i.e. an extension of degree $4$ and exponent $1$.
However, the only exponent $1$ subextension is $F(\sqrt{z})$.  We show this
for the reader's convenience.  Let $\zeta=\sqrt[4]{z}$, $\chi=\sqrt{x\sqrt{z}+y}$
and let $t\in E$ with $t^2\in F$.  Then there are $a_i,b_i\in F$, $0\leq i\leq 3$,
such that
$$t=\sum_{i=0}^3a_i\zeta^i+\chi\sum_{i=0}^3b_i\zeta^i,$$
and we get
$$t^2=a_0^2+a_2^2z+(b_3z)^2x+b_1^2xz+b_0^2y+b_2^2yz+
(a_1^2+a_3^2z+b_0^2x+b_2^2xz+b_1^2y+b_3^2yz)\sqrt{z}\in F,$$
so we get 
$$a_1^2+a_3^2z+b_0^2x+b_2^2xz+b_1^2y+b_3^2yz =0$$
which, by $2$-independence, implies $a_1=a_3=b_0=b_1=b_2=b_3=0$
and thus, $t=a_0+a_2\sqrt{z}\in  F(\sqrt{z})$.

Now let $\pi\cong\pff{x,y,c_1,\ldots,c_n}_b$ be as in Example \ref{ex1}.
Note that by general theory, 
the $2$-independence of $\{ z,x,y,c_1,\ldots, c_n\}$ over $F$
implies the $2$-independence of $\{ \sqrt{z},x,y,c_1,\ldots, c_n\}$
over $L=F(\sqrt{z})$.  Hence $\pi_L$ is anisotropic.  Since 
$E=L(\sqrt{\sqrt{z}},\sqrt{x\sqrt{z}+y})$, we see that we are
in the same situation as in Example \ref{ex1} but with $F$ replaced by
$L$ and $z$ replaced by $\sqrt{z}$.  This shows on the one hand that
$\pi_E$ is isotropic, and on the other hand that if $K$ is any quadratic
extension of $L$ contained in $E$, then $\pi_K$ is anisotropic.

We have that $\pi$ is isotropic over $E$ but anisotropic over any
extension of $F$ properly contained in $E$.
Indeed, if $K$ is any extension of $F$ properly contained in $E$, then $K=F$,
or $[K:F]=2$ in which case $K=L$, or  $[K:F]=4$ in which case
$K$ will contain a quadratic extension of $F$ which must be
$L$ and $K/L$ will thus be a quadratic extension.  By the above,
$\pi_K$ is anisotropic in all these cases.
\qed\end{exam}

\section{Quadratic forms over purely inseparable field extensions}
The isotropy behaviour of quadratic forms over exponent $1$ extensions has been
studied in \cite{h2}, including a determination of the Witt kernel for such
extensions, i.e., the classifiction of quadratic forms that become hyperbolic
over such extensions.  Complete results for quartic extensions can be found
in \cite{hs}, and the determination of Witt kernels for arbitrary purely
inseparable extensions can be found in \cite{so}, \cite{alo}.
We will not need the full thrust of these results but instead
we will focus primarily on  quadratic Pfister forms
and some explicit examples that have not been exhibited before in the
literature in the way we require. 

Let us remark at this point that $1$-fold quadratic Pfister forms
are of little interest in our context.  Indeed, if
$\pi\cong\qpf{a}\cong [1,a]\in P_1F$ is anisotropic and $E/F$ is any
field extension, then $\pi_E$ is isotropic iff $a\in\wp(E)=\{ e^2+e\,|\,e\in E\}$
iff the separable quadratic extension $F(\wp^{-1}(a))$
is contained in $E$ (where $\wp^{-1}(a)$ denotes a root
of the separable irreducible polynomial $X^2+X+a\in F[X]$).
In particular, $\pi$
will not become isotropic over any purely inseparable extension.

The isotropy behaviour of quadratic forms over quadratic p.i.\ extensions is
quite well understood.

\begin{lem}\label{qf-quad}
Let $q$ be an anisotropic quadratic form over $F$ and let $K=F(\sqrt{a})$,
$a\in F\setminus F^2$, be a quadratic p.i.\ extension.  Then $q_K$ is isotropic
iff $c\qf{1,a}\prec q$ for some $c\in F^*$.  In this case, $\pff{a}_b\otimes q$ is isotropic.
\end{lem}
\begin{proof}
The equivalence is well known (see, e.g., the proof of \cite[Lemma 5.4]{hl2}).
Furthermore, if $c\qf{1,a}\prec q$ then
$c\pff{1,a}\cong c\qf{1,1,a,a}\prec \pff{a}_b\otimes q$, and $\qf{1,1,a,a}$
is obviously isotropic.
\end{proof}

\begin{prop}\label{qf-exp1}
Let $E/F$ be an exponent $1$ p.i.\ extension and let $\pi\in P_2F$
be anisotropic.  If $\pi_E$ is isotropic (and hence hyperbolic), there exist
$a,c\in F^*$ such that $\sqrt{a}\in E$ and $\pi\cong\qpf{a,c}$.  In particular,
$\pi$ becomes isotropic over a quadratic p.i.\ extension of $F$ contained in $E$.
\end{prop}

\begin{proof}  The proof can be extracted from the proof of \cite[Th.~3.4]{h2},
but in our special situation, the argument can be condensed considerably and
we include it for the reader's convenience.
Note that by assumption, $E^2\subseteq F$.  Let $\pi\cong\qpf{u,v}\in P_2F$
be anisotropic but hyperbolic over $E$.  Then any Pfister neighbor
of $\pi$ over $F$ will also become isotropic
over $E$.  Now $[1,u]\perp\qf{v}\prec\pi\cong [1,u]\perp v[1,u]$, so there
exist $x,y,z\in E$, not all equal to $0$, with
$$0=x^2+xy+uy^2+vz^2.$$  
Then at least two of the $x,y,z$ must be nonzero.
If, say, $x\neq 0$, we may divide by $x^2$ und thus, we may assume
without loss of generality that
we have an equation
$$1+y+uy^2+vz^2=0,\ y,z\in F.$$
But $y^2,z^2\in F$ and thus also $y\in F$.  But then, the anisotropy over $F$
implies that $z\in E\setminus F$.  With $z^2=a\in F$ it follows that
$\pi$ becomes hyperbolic over $F(\sqrt{a})$ from which we can conclude that
$\pi\cong\qpf{a,c}$ for some $c\in F$ (see, e.g., \cite[Cor.~2.8]{a}).
The case $y\neq 0$ is similar.
\end{proof}

We will see later on that this result does not generalize to
$n$-fold quadratic Pfister forms for $n\geq 3$.  In the case of 
$n$-fold quadratic Pfister forms ($n\geq 2$), the result will also not generalize
to p.i.\ extensions of higher exponent, not even to simple such extensions
(in stark contrast to the case of bilinear Pfister forms!) as the next
examples will show.

In the construction of our examples, we will often use generic methods.
In particular, we will work over the field of Laurent series $F(\!(t)\!)$.
Now theories of quadratic forms over valued fields in charactersitic $2$ can
be found in the literature, but we will not need a full fledged
such theory but rather
only define some concepts and use some facts that will be needed
in our constructions to keep the paper as self-contained as possible.  

We call $a\in F(\!(t)\!)$ a unit if $a\in F[[t]]^*=F^*+tF[[t]]$.  For our purposes,
we call a quadratic form $q$ over $F(\!(t)\!)$ quasi-unimodular if $\phi$ has 
a representation 
$$\phi\cong [a_1,b_1]\perp \ldots\perp [a_r,b_r]\perp\qf{c_1,\ldots,c_s},\quad
a_i,b_i,c_j\in F[[t]]^*\cup\{ 0\}$$
(we include $0$ to make sure that the hyperbolic plane $[0,0]$ will be 
quasi-unimodular).  Then $a_i\in \alpha_i+tF[[t]]$ with some uniquely
determined $\alpha_i\in F$
(and similarly, $b_i\in \beta_i+tF[[t]]$, $c_j\in \gamma_j+tF[[t]]$ with
$\beta_i,\gamma_j\in F$).  Then we get a quadratic form over $F$ defined by
$$\ovl{\phi}=[\alpha_1,\beta_1]\perp
\ldots\perp [\alpha_r,\beta_r]\perp\qf{\gamma_1,\ldots,\gamma_s}$$
that we call the residue form for this representation.  
The following result is fairly straightforward.
\begin{lem}\label{residue}
Let $\phi$ be a quasi-unimodular quadratic form over $F$ and let $\ovl{\phi}$
be the residue form of some quasi-unimodular representation of $\phi$.
If $\ovl{\phi}$ is anisotropic then $\phi$ is anisotropic.   In this situation,
if $\ovl{\phi}'$ is the residue form of any other quasi-unimodular representation
of $\phi$, then $\ovl{\phi}\cong\ovl{\phi}'$ over $F$.
\end{lem}

\begin{prop}\label{prop-t-aniso}
Let $\rho$, $\sigma$, $\alpha$, $\beta$ be quasi-unimodular quadratic forms
over $F(\!(t)\!)$.  Assume that $\alpha$, $\beta$ 
are totally singular of the same dimension $n$ having
quasi-uni\-mod\-ular representations
$\alpha\cong\qf{a_1,\ldots,a_n}$, $\beta\cong\qf{b_1,\ldots,b_n}$.
Let
$$\begin{array}{rcl}
\psi & \cong &  \rho\perp\alpha,\\
\tau & \cong &  \sigma\perp\beta,\\
\phi & \cong &  \rho\perp t^{-1}\sigma\perp [a_1,t^{-1}b_1]\perp\ldots [a_n,t^{-1}b_n].
\end{array}$$
If $\ovl{\psi}$ and $\ovl{\tau}$ are anisotropic over $F$
then $\phi$ is anisotropic over $F(\!(t)\!)$.
\end{prop} 
\begin{proof}  The proof uses a fairly standard constant coefficient
argument which we will only
sketch.  Suppose $\phi$ is isotropic. 
Let $\dim\rho=r$, $\dim\sigma=s$. By multiplying
by a suitable power of $t$,
one may assume that there are $e_i,f_j,g_k,h_k \in F[[t]]$
($1\leq i\leq r$, $1\leq j\leq s$, $1\leq k\leq n$), not all divisible by $t$,
such that
$$0=\rho(e_1,\ldots,e_r)+t^{-1}\sigma(f_1,\ldots,f_s)+
\sum_{k=1}^n\bigl(a_kg_k^2+g_kh_k+t^{-1}b_kh_k^2\bigr).$$
By considering the constant terms of the involved power series
and by computing the coefficient of $t^0$ resp. $t^{-1}$ in the above
equation, one readily gets that we find an equation
$$\ovl{\rho}(\epsilon_1,\ldots,\epsilon_r)+\ovl{\alpha}(\delta_1,\ldots,\delta_n)=0
\quad\mbox{with $\epsilon_i,\delta_k\in F$, not all equal to zero,}$$
or 
$$\ovl{\sigma}(\gamma_1,\ldots,\gamma_s)+\ovl{\beta}(\mu_1,\ldots,\mu_n)=0
\quad\mbox{with $\gamma_j,\mu_k\in F$, not all equal to zero,}$$
showing that $\ovl{\psi}$ or $\ovl{\tau}$ is isotropic over $F$.
\end{proof}

\begin{cor}\label{cor-t-aniso} Let $a,a_1,\ldots,a_n\in F[[t]]^*$
such that $b\cong\qf{a_1,\ldots,a_n}_b$ is an anisotropic
bilinear form over $F(\!(t)\!)$ and with
associated totally singular quasi-unimodular quadratic form
$q_b\cong\qf{a_1,\ldots,a_n}$.  If $\ovl{q_b}$ is anisotropic over $F$
then $b\otimes [1,at^{-1}]$ is anisotropic over $F(\!(t)\!)$.
\end{cor}
\begin{proof}
Note that $b$ is anisotropic iff the associated quadratic form
$q_b$ is anisotropic.
Also,
$$b\otimes [1,at^{-1}]\cong\bbperp_{i=1}^na_i[1,at^{-1}]
\cong\bbperp_{i=1}^n[a_i,aa_i^{-1}t^{-1}],$$
and we can apply Proposition \ref{prop-t-aniso} to
$\psi\cong\alpha$ and 
$$\tau\cong\beta\cong a\qf{a_1^{-1},\ldots,a_n^{-1}}
\cong a\qf{a_1,\ldots,a_n}\cong\alpha.$$
\end{proof}

In the next example, we extend Example \ref{ex1} by showing
that to each $n\geq 0$, there are examples of fields over which
there exist anisotropic $(n+3)$-fold quadratic Pfister forms that
become isotropic over a biquadratic p.i.\ extension but stay anisotropic
over any quadratic subextension.

\begin{exam}\label{ex1a}
Let $F$, $E$, $\pi$ be as in Example \ref{ex1}, put $L=F(\!(t)\!)$, 
$$M=L(\sqrt{z},\sqrt{xz+y})=F(\sqrt{z},\sqrt{xz+y})(\!(t)\!)=E(\!(t)\!)$$
and
$q=\pi\otimes [1,t^{-1}]\in P_{n+3}L$.  Obviously, $\pff{z,x,y,c_1,\ldots,c_n}$
is quasi-unimodular with anisotropic residue form (which essentially
is the form with the same coefficients), and by Corollary \ref{cor-t-aniso},
it follows that $\pff{z}_b\otimes q$ is anisotropic and thus also $q$.
But as in  Example \ref{ex1}, $\pi_E$ and hence $\pi_M$ and $q_M$ are isotropic.

Let $K$ be any quadratic extension of $L$ contained in $M$.  As in
Example \ref{ex1}, we may write $K=L(a)$ where
$$a=u\sqrt{z}+v\sqrt{xz+y}+w\sqrt{z^2x+yz}$$
with $u,v,w\in L$, not all
equal to $0$.  Multiplying by a suitable power of $t$, we may
furthermore assume that $u=u_0+u'$, $v=v_0+v'$, $w=w_0+w'$
with $u_0,v_0,w_0\in F$, not all equal to $0$,
and $u',v',w'\in tF[[t]]$.  Let 
$$\ovl{a}=u_0\sqrt{z}+v_0\sqrt{xz+y}+w_0\sqrt{z^2x+yz}.$$
Then $a^2\in \ovl{a}^2+tF[[t]]$.
Using the $2$-independence of $x,y,z$, and the fact that we do not have
$u_0=v_0=w_0=0$, we have $\ovl{a}^2\neq 0$ and thus 
$a^2\in F[[t]]^*$.  As in Example \ref{ex1},
we get (over $F$!)
$$\pff{\ovl{a}^2,x,y,c_1,\ldots,c_n}\cong \pff{z,x,y,c_1,\ldots,c_n}$$
which is anisotropic over $F$.  

By the above, $\pff{a^2,x,y,c_1,\ldots,c_n}$ is quasi-unimodular
with anisotropic residue form $\pff{\ovl{a}^2,x,y,c_1,\ldots,c_n}$.
By Corollary \ref{cor-t-aniso}, it follows that
$$\pff{a^2,x,y,c_1,\ldots,c_n}_b\otimes [1,t^{-1}]\cong
\pff{a^2}_b\otimes q$$
is anisotropic over $L$.  Therefore, with $L(a)=K$ and by 
Lemma \ref{qf-quad}, $q_K$ is anisotropic.
\qed\end{exam}

\begin{exam}\label{ex1b}
In the previous example,  we have constructed fields $L$ over
which there exist an anisotropic $(n+3)$-fold Pfister form $\pi$ ($n\geq 0$)
over $L$ and a p.i.\ extension
$M$ of $L$ of degree $4$
such that $\exp_2(M/L) = 1$ and $\pi$ becomes isotropic over $M$
but stays anisotropic over
any {\em simple} extension $K/L$ contained in $E$ (since in
that example, $M/L$ is p.i.\ biquadratic,
any such simple subextensions  will be p.i.\ quadratic).

One can easily generalize this as follows:

\medskip

\noindent{\em 
For any integers $n\geq 0$, $m\geq 2$ and $\ell$ with $1\leq \ell\leq \max\{1,m-2\}$
there exist 
a field $L$, an anisotropic $(n+3)$-fold Pfister form $\pi$ over $L$ and a 
p.i.\ extension $M$ of $L$ of degree $2^m$  and exponent $\exp_2(M/L) = \ell$
such that $\pi$ becomes isotropic over $M$ but $\pi$ stays anisotropic
over any simple extension $K$ of $L$ contained in $M$.}

\medskip

Indeed, choose integers $r\geq 0$ and $1\leq m_1,\ldots,m_r\leq \ell$ such that
$$m=2+m_1+\ldots +m_r\quad\mbox{and}\quad \mbox{$m_1=\ell$ if $m\geq 3$}$$
(this is always possible under the assumptions).
Assume furthermore that $L=F(\!(t)\!)$ where this time $F$
has a $2$-independent set
$$\{ z,x,y,c_1,\ldots,c_n,b_1,\dots,b_r\}.$$
Again, we choose $\pi\cong\qpf{x,y,c,\ldots,c_n,t^{-1}}$ over $L$
This time, we define
$$\begin{array}{rcl}
F' & = & F\bigl(\sqrt[2^{m_1}]{b_1},\ldots,\sqrt[2^{m_r}]{b_r}\bigr)\\[1ex]
L' & = & F'(\!(t)\!)\\[1ex]
M & = & L\bigl(\sqrt{z},\sqrt{xz+y},\sqrt[2^{m_1}]{b_1},
\ldots,\sqrt[2^{m_r}]{b_r}\bigr)
\end{array}$$
Note that
$$M=L'(\sqrt{z},\sqrt{xz+y})=F'(\sqrt{z},\sqrt{xz+y})(\!(t)\!).$$
By $2$-independence, one easily checks that the p.i.\ extension $M/L$
satisfies the required properties regarding degree and exponent.
Furthermore, by the general theory of $2$-independence,
$$\bigl\{ z,x,y,c_1,\ldots,c_n,\sqrt[2^{m_1}]{b_1},
\ldots,\sqrt[2^{m_r}]{b_r}\bigr\}$$
is a $2$-independent set over $F'$ (and over $L'$),
thus, we can now proceed exactly as in Example \ref{ex1a} and conclude
that $\pi_{L'}$ is anisotropic, $\pi_M$ is isotropic, and $\pi_{K'}$ is 
anisotropic over any simple extension $K'$ of $L'$ contained in $M$.

Now let $K=L(w)$ be any simple extension of $L$ contained in $M$.  Then
obviously $K':=KL'=L'(w)$ is a simple extension of $L'$ contained in $M$,
hence $\pi_{K'}$ is anisotropic and thus also $\pi_K$ since 
$K\subseteq K'$.
\qed\end{exam}

\begin{exam}\label{ex2a}
Let $F$, $E$, $\pi$ be as in Example \ref{ex2}, put $L=F(\!(t)\!)$, 
$${\textstyle M=L(\sqrt[4]{z},\sqrt{x\sqrt{z}+y})=
F(\sqrt[4]{z},\sqrt{x\sqrt{z}+y})(\!(t)\!)=E(\!(t)\!)}$$ and
$q=\pi\otimes [1,t^{-1}]\in P_{n+3}L$.
As in Example \ref{ex2a}, $M/L$ is nonmodular of degree $8$,
and the only quadratic extension of $L$
inside $M$ is $L(\sqrt{z})=F(\sqrt{z})(\!(t)\!)$.  
As in Example \ref{ex1a}, 
the $2$-independence of $\{ \sqrt{z},x,y,c_1\ldots,c_n\}$ over $L(\sqrt{z})$
implies the anisotropy of $q_{L(\sqrt{z})}$, and  $M/L(\sqrt{z})$ 
is a biquadratic extension.  Combining the 
arguments in Example \ref{ex1a} with those in Example \ref{ex2} shows that
if $K$ is any extension of $L$ properly contained in $M$ then $q_K$ is 
anisotropic.  We leave the details to the reader.
\qed\end{exam}

We now want to construct examples of anisotropic $(n+2)$-fold Pfister forms
over a field $F$ that become isotropic over a simple p.i.\ extension
$E/F$ of exponent $m\geq 2$ but which stay anisotropic over any subextension
of exponent $\leq m-1$.  For this, we collect two lemmas that we will need
in our construction.

The first lemma is well known (and easy to prove).  We assume as usual
that we have characteristic $2$, but of course it holds more generally
in any positive characteristic $p$.
\begin{lem}\label{filter}  Let $F$ be a field and $m\geq 1$ be an integer.
Let $x\in F\setminus F^2$ and $E=F\bigl(\sqrt[2^m]{x}\bigr)$.  Then
$[E:F]=2^m$ and all the intermediate fields are given in the following filtration
of successive quadratic p.i.\ extensions:
$$F\subset F\bigl(\sqrt{x}\bigr)\subset  F\bigl(\sqrt[4]{x}\bigr)\subset\ldots
\subset F\bigl(\sqrt[2^{m-1}]{x}\bigr)\subset F\bigl(\sqrt[2^m]{x}\bigr)=E.$$
\end{lem}

The next lemma provides some relations for quadratic forms over fields of
characteristic $2$.
\begin{lem}\label{relation} Let $u,v\in F^*$ and $m\geq 1$ be an integer.
\begin{enumerate}
\item[(i)]  $\qpf{u,uv}\cong\qpf{v,uv}$.  In particular, $\qpf{u,uv}$ will be
isotropic over $F(\sqrt{v})$.
\item[(ii)] $[1,u^{2^m}]\cong [1,u]$.
\end{enumerate}
\end{lem}
\begin{proof}
(i) We have 
$$\qpf{u,uv}\cong [1,uv]\perp u[1,uv]\cong [1,uv]\perp [u,v]\cong 
[1,uv]\perp v[1,uv]\cong \qpf{v,uv}.$$
Therefore, $\qf{1,v}\prec\qpf{u,uv}$ will be isotropic over $F(\sqrt{v})$.

(ii) This follows since $u^{2^m}\equiv u\bmod\wp(F)$ (essentially, this is a 
comparison of the so-called Arf-invariants of the two forms).
\end{proof}

\begin{exam}\label{ex3}
Let $F$ be a field (of characteristic $2$ as usual) with 
elements
$a_1,\ldots,a_n$, $x,y\in F^*$ ($n\geq 0$), let $m\geq 2$ be an integer
and let
$$\xi=\sqrt[2^{m-1}]{x},\ E'=F(\xi)\ \mbox{and}\ 
E=E'(\sqrt{\xi})=F\bigl(\sqrt[2^{m}]{x}\bigr).$$
We define the
Pfister form 
$$\pi\cong\qpf{a_1,\ldots,a_n,y,y^{2^{m-1}}x}.$$
Note that $\pi$ will be isotropic over $E$ since over $E'$,
we have by Lemma \ref{relation}(ii) that
$$\qpf{y,y^{2^{m-1}}x}_{E'}\cong \qpf{y,(y\xi)^{2^{m-1}}}_{E'}
\cong  \qpf{y,y\xi}_{E'}$$
which is isotropic over $E=E'(\sqrt{\xi})$ by Lemma \ref{relation}(i).

If we now can find a field $F$ with elements $a_i,x,y$ as above 
such that $[E:F]=2^m$ (or, equivalently, $x\in F\setminus F^2$) and
$\pi_{E'}$ anisotropic, then we have our example:
$\pi$ will be an anisotropic $(n+2)$-fold Pfister form
over $F$ that will become isotropic over a simple p.i.\ extension
$E$ with $\exp_2(E/F)=m$, but, by Lemma \ref{filter}, $\pi$ will stay
anisotropic over any extension  of $F$ of exponent $\leq m-1$ that is
contained in $E$.

For this, let $F_0$ be a field with $2$-independent elements
$a_1,\ldots,a_n,y\in F_0^*$ and put
$$E=F_0(\!(t)\!)\ \mbox{and}\ F=F_0(\!(t^{2^m})\!)\ .$$
Furthermore, put $x=t^{-2^m}$.  Clearly, $E=F(\sqrt[2^m]{x})$.
We put $\tau=t^2$ so that the above $\xi$ will become $\xi=t^{-2}=\tau^{-1}$.
We thus get the filtration
$$F=F_0(\!(t^{2^m})\!)\subset F_0(\!(t^{2^{m-1}})\!)\subset\ldots
\subset F_0(\!(\tau)\!)=E'\subset F_0(\!(t)\!)=E'(\sqrt{\xi})=E$$
of all intermediate fields between $F$ and $E$ (cf.\ Lemma \ref{filter}).

Since $a_1,\ldots,a_n,y\in F_0^*$ are $2$-independent, 
the form $\pff{a_1,\ldots,a_n,y}_b$ is anisotropic over $F_0$.
Also, over $E'=F_0(\!(\tau)\!)$ and by Lemma \ref{relation}(ii), we have
$$\bigl[1,y^{2^{m-1}}x\bigr]\cong \bigl[1,(y\tau^{-1})^{2^{m-1}}\bigr]\cong
\bigl[1,y\tau^{-1}\bigr],$$
we can apply Corollary \ref{cor-t-aniso} to conclude that
$$\pi_{E'}\cong \pff{a_1,\ldots,a_n,y}_b\otimes \bigl[1,y^{2^{m-1}}x\bigr]
\cong \pff{a_1,\ldots,a_n,y}_b\otimes \bigl[1,y\tau^{-1}\bigr]$$
is anisotropic over $E'$
\qed\end{exam}

\begin{exam}\label{ex3a}
In the previous example, we have constructed a field $F$ with
an anisotropic $(n+2)$-fold Pfister form, $n\geq 0$, and a simple p.i.\ extension
$E/F$ with $[E:F]=2^m$ and $\exp_2(E/F)=m\geq 2$ such that
$\pi$ becomes isotropic over $E$ but $\pi$ will stay anisotropic
over any extension $K/F$ of exponent $\leq m-1$ contained in $E$.

Again, this can be generalized as follows: 

\medskip 

\noindent{\em Given any integers $n\geq 0$,
$m\geq \ell\geq 2$, there exists
a field $F$ with an anisotropic $(n+2)$-fold Pfister form $\pi$ over $F$
and with a p.i.\ extension $M/F$ with $[M:F]=2^m$ and $\exp_2(M/F)=\ell$
such that $\pi_M$ is isotropic, but $\pi_K$ is anisotropic for any
extension $K$ of $F$ contained in $M$ with $\exp_2(K/F)\leq \ell -1$.}

\medskip

Indeed, let $r\geq 0$ and $1\leq m_1,\ldots,m_r\leq \ell$ be integers with
$$m=\ell+m_1+\ldots +m_r$$
(this is always possible under the assumptions).
Let us proceed as in Example \ref{ex3} by choosing a field
$F_0$ but this time with a $2$-independent set 
$$\{ y,a_1,\ldots,a_n,b_1,\ldots,b_r\},$$
we put $F=F_0(\!(t^{2^\ell})\!)$ and $x=t^{-2^\ell}$
and $\pi\cong\qpf{a_1,\ldots,a_n,y,y^{2^{\ell-1}}x}$.
Let 
$$\begin{array}{rcl}
F_0' & = & F_0\bigl(\sqrt[2^{m_1}]{b_1},\ldots,\sqrt[2^{m_r}]{b_r}\bigr)\\[1ex]
F' & = & F_0'(\!(t^{2^\ell})\!)\\[1ex]
M & = & F\bigl(\sqrt[2^\ell]{x},\sqrt[2^{m_1}]{b_1},
\ldots,\sqrt[2^{m_r}]{b_r}\bigr)
\end{array}$$
Note that
$M=F'\bigl(\sqrt[2^\ell]{x}\bigr)$.
Using $2$-independence,
one readily checks that the p.i.\ extension $M/F$ satisfies
$[M:F]=2^m$ and $\exp_2(M/F)=\ell$.

Since 
$${\textstyle \bigl\{ y,a_1,\ldots,a_n,
\sqrt[2^{m_1}]{b_1},\ldots,\sqrt[2^{m_r}]{b_r}\bigr\}}$$
is a $2$-independent set over $F_0'$ (cf. Example \ref{ex1b}), we conclude
as in Example \ref{ex3} that $\pi_M$ will be isotropic,
but $\pi_{K'}$ will be anisotropic
over any extension $K'$ of $F'$ contained in $M$ with $\exp_2(K'/F')\leq \ell-1$
(note that in Example \ref{ex3} we are in the case $m=\ell$ and $r=0$).

Now let $K$ be any extension of $F$ contained in $M$ with $\exp_2(K/F)\leq \ell -1$.
Then clearly $\exp_2(KF'/F')\leq \ell -1$, so by the above, $\pi$ stays
anisotropic over $KF'$ and thus over $K$.
\qed\end{exam}

\section{Applications to quaternions and octonions}
We refer to \cite{sv} concerning basic properties of quaternions and octonions
that we will state in the sequel.
Let $a,b\in F^*$, $c\in F$. 
Recall that a quaternion algebra $Q=(b,c]_F$
over such a field $F$ of characteristic $2$ is a $4$-dimensional
central simple (associative) algebra generated by two elements
$i,j$ subject to the relations
$$i^2+i=c, j^2=b, ji=(i+1)j.$$
$Q$ carries an $F$-linear involution $\sigma$ given by $\sigma(i)=i+1$, 
$\sigma(j)=j$. 
An octonion algebra $O=(a,b,c]_F$ over $F$ 
is an $8$-dimensional nonassociative composition algebra that
can be gotten from a quaternion algebra $Q$ as above
through the Cayley-Dickson construction by defining
$O=Q\oplus Q\ell$ where the product is given by
$$(x\oplus y\ell)(z\oplus w\ell)=
(xz+a\sigma(w)y)\oplus (wx+y\sigma(z))\ell,\ x,y,z,w\in Q.$$
$O$ carries an involution $\tau$ given by
$$\tau(x\oplus y\ell)=\sigma(x)\oplus y\ell.$$ 
On $Q$ (resp. $O$) one has a norm given by $N_Q(x)=x\sigma(x)$ (resp.
$N_O(x)=x\tau(x)$) which defines a quadratic Pfister form
on $Q$ (resp. $O$) given by 
$\eta_Q\cong\qpf{b,c}$ (resp. $\eta_O\cong\qpf{a,b,c}$).

It is well known that the isometry class of the norm form determines
the isomorphism class of the quaternion (resp. octonipn algebra).
That means, if $A_i$, $i=1,2$, are quaternion (resp. octonion) algebras with
respective norm forms $\eta_i=\eta_{A_i}$, then $A_1\cong A_2$ as
algebras iff $\eta_1\cong\eta_2$ as quadratic forms.  Furthermore,
$A_i$ is division iff $\eta_i$ is anisotropic.  

Using this correspondence
between the division property of quaternion/octonion algebras and
the anisotropy of their associated norm forms and by invoking our 
Examples \ref{ex1b} and \ref{ex3a}, we now give our version
of the negative answer to the question posed by M\"uhlherr and Weiss
from the introduction.

\begin{cor}
To any integers $m\geq 2$ and $\ell$ with $1\leq \ell\leq\max\{1,m-2\}$
(resp.\ $2\leq \ell\leq m$),
there exists a field $F$ of characteristic $2$ together with a 
p.i.\ extension $M/F$ of degree $2^m$ and exponent $\ell$,
and with an octonion division algebra $O$
(resp.\ a quaternion division algebra $Q$ and an octonion division algebra
$O$) over $F$ such that $O_M$ is split (resp.\ $Q_M$ and $O_M$ are split),
but $O$ (resp.\ $Q$ and $O$) will stay division over any simple
(resp.\ exponent $\leq \ell-1$) extension $K$ of $F$
contained in $M$.
\end{cor}

\end{document}